\documentclass[12pt]{amsart}
\usepackage[pdftex]{graphicx}
\usepackage{amsmath, amssymb}
\usepackage{tikz}
\usepackage{subcaption}

\theoremstyle{plain}
\newtheorem{theorem}{Theorem}[section]

\newtheorem{lemma}[theorem]{Lemma}
\newtheorem{corollary}[theorem]{Corollary}
\theoremstyle{definition}
\newtheorem{mydef}[theorem]{Definition}

\providecommand{\Irr}{\textnormal{Irr}}
\providecommand{\cd}{\textnormal{cd}}

\begin{document}

\title{}
\author{}

\begin{abstract}
We define two different simplicial complexes, the common divisor simplicial complex and the prime divisor simplicial complex, from a set of integers, and explore their similarities. We will define a map between the two simplicial complexes, and use this map to show that for any set of integers, the fundamental groups of the resulting simplicial complexes are isomorphic. 
\\{\scshape Keywords:} Abstract Algebra, Topology, Algebraic Topology, Fundamental Group, Simplicial Complex, simplex
\end{abstract}

\begin{center}
{\large\scshape Fundamental Groups of Simplicial Complexes}\par\bigskip

E. Wheeler\\[5pt]
Carthage College\\[5pt]
ewheeler2$@$carthage.edu\\[5pt]
May 30, 2016\\[5pt]
\maketitle
\end{center}

\section{Introduction}
Suppose $G$ is a finite group.  Historically, much has been deduced about the structure of $G$ when only given information about its irreducible characters. Notationally, we write $\Irr(G)$ for the set of irreducible characters of a group $G$ and we write $\cd(G) = \{\chi(1) \, \, | \, \, \chi \in \Irr(G)\}$.  An enormous amount of information can be deduced about the group $G$ when only given the set $\cd(G)$.  For a comprehensive overview, see \cite{lewis2008overview}.  

Two useful ways of visualizing the members of the set $\cd(G)$ have frequently been employed in this area of research.  The first is called the \emph{character degree graph} of $\cd(G)$, denoted by $\Gamma(G)$.  The vertices of this graph are the members of the set $\cd(G) \setminus \{1\}$, and there is an edge connecting two vertices if the corresponding irreducible character degrees have a nontrivial common divisor.  The second is the \emph{prime vertex graph}, denoted $\Delta(G)$, which has the primes dividing some member of $\cd(G)$ as its vertices, and there is an edge between two vertices if there is a member of $\cd(G)$ divisible by the two associated primes.  More recent research questions have involved what can be said about the structure of $G$ given only either $\Gamma(G)$ or $\Delta(G)$; \cite{lewis2008overview} also contains a summary of these types of results.

It has long been understood that there is an intimate relationship between the graphs $\Gamma(G)$ and $\Delta(G)$.  For example, it is known (\cite{lewis2008overview}) that $\Gamma(G)$ is connected if and only if $\Delta(G)$ is connected, and that the distance between two vertices in $\Gamma(G)$ is closely related to the distance between ``corresponding'' vertices in $\Delta(G)$.  The relationship between $\Gamma(G)$ and $\Delta(G)$ has given researchers the fluidity to obtain results pertaining to only one of these graphs and use the proper correspondence to obtain a result about the other.

Still more current research \cite{jensen2015character} suggests moving away from studying only the structure of the common divisor graph of $G$ and the prime divisor graph of $G$ toward the study of the \emph{common divisor simplicial complex} of $G$, denoted $\mathcal{G}(G)$, and the \emph{prime vertex simplicial complex} of $G$, henceforth $\mathcal{D}(G)$.  In \cite{jensen2015character}, the author works primarily with the simplicial complex $\mathcal{G}(G)$ obtaining results about the fundamental group.  It is unclear in this work if there is any analogous result regarding the fundamental group of $\mathcal{D}(G)$ as no more general correspondence between $\mathcal{G}(G)$ and $\mathcal{D}(G)$ is known.  

Although the connection between $\Gamma(G)$ and $\Delta(G)$ is noticeable and documented, the nature of the relationship between the two graphs is imprecise.  Moreover, as work with $\mathcal{G}(G)$ and $\mathcal{D}(G)$ is just emerging, there is no current research justifying a more general relationship between these two structures.  In this paper, we establish the precise nature of the relationship between $\mathcal{G}(G)$ and $\mathcal{D}(G)$ by creating a specific map $\eta$ between the two simplicial complexes.  The map $\eta$, when restricted to $\Gamma(G)$ and $\Delta(G)$, makes precise the relationship so frequently used in the works summarized in \cite{lewis2008overview}.  We then proceed to establish that if $G$ is a finite group with $\mathcal{G}(G)$ connected, then $\eta$ induces a map $\eta_*$ on the fundamental groups of $\mathcal{G}(G)$ and $\mathcal{D}(G)$, denoted $\pi_1(\mathcal{G}(G))$ and $\pi_1(\mathcal{D}(G))$, respectively.  Our main result is the following.

\begin{theorem}
Suppose $G$ is a finite group with $\mathcal{G}(G)$ connected.  The induced map $\eta_*$ from $\pi_1(\mathcal{G}(G))$ to $\pi_1(\mathcal{D}(G))$ is a group isomorphism. 
\end{theorem}

This correspondence allows us to derive analogous results to those found in  \cite{jensen2015character} for $\mathcal{D}(G)$, and it opens the door for further analogies to be drawn regarding all works involving $\pi_1(\mathcal{G}(G))$ or $\pi_1(\mathcal{D}(G))$.  Having the more general map $\eta$ between $\mathcal{G}(G)$ and $\mathcal{D}(G)$ allows us the ability to attempt to translate all kinds of results established on either $\mathcal{G}(G)$ or $\mathcal{D}(G)$ to the other simplicial complex in a routine way. 

For the majority of the paper, results are stated more generally for the common divisor simplicial complex and prime divisor simplicial complex of a set of positive integers $X$. In Section \ref{prelims}, we provide the necessary topological definitions for our work. In Section \ref{correspondence}, we define the map $\eta$, establishing a correspondence between $\mathcal{G}(X)$ and $\mathcal{D}(X)$. In Section \ref{inducedmap}, we show that $\eta$ induces a map $\eta_*$ on edge-paths in $\mathcal{G}(X)$ and $\mathcal{D}(X)$, and show that this induced map is an isomorphism between $\pi_1(\mathcal{G}(X))$ and $\pi_1(\mathcal{D}(X))$. In Section \ref{applications}, we apply the general theory from the previous sections to group theory. In Section \ref{research}, we suggest some directions for future research on this topic. And in Section \ref{example}, we have a useful example of simplices and simplicial complexes. The author would like to thank Sara Jensen for her help in writing the introduction to this paper.  

\section{Preliminary Definitions}\label{prelims}
In this Section, we begin with a few preliminary definitions necessary for the rest of this paper. 
\begin{mydef}
A \textbf{simplex} is the generalization of a triangle to arbitrary dimensions. An $n$-simplex is the convex hull of $n+1$ vertices in $\mathbb{R}^n$.
\end{mydef}
A 0-simplex is a point, a 1-simplex is a line, a 2-simplex is a triangle, and a 3-simplex is a tetrahedron. Simplices can also exist in more than three dimensions, where they take the shape of the convex hull of their vertices. An $n$-simplex can equivalently be described as a simplex of dimension $n$. Any $n$-simplex, because it has $n+1$ vertices, can be embedded in $\mathbb{R}^{n+1}$ by assigning each vertex an axis and giving it a coordinate of 1 on that axis and 0 on every other axis. 
\begin{mydef}
A \textbf{face} of a simplex is any subset of the vertices of that simplex. 
\end{mydef}
\begin{mydef}
An (abstract) \textbf{simplicial complex} $\mathcal{K}$ is a pair $(\mathcal{V}, \mathcal{S})$ where $\mathcal{V}$ is a finite set of elements called vertices and $\mathcal{S}$ is a set of nonempty subsets of $\mathcal{V}$ called simplices such that all singleton subsets of $\mathcal{V}$ belong to $\mathcal{S}$ and if $\sigma\in\mathcal{S}$ and $\sigma'\subseteq\sigma$ then $\sigma'\in\mathcal{S}$.
\end{mydef}
A simplicial complex can be embedded in $\mathbb{R}^n$ where $n$ is the number of vertices in the simplicial complex. In this way we define the fundamental group, using the standard topology on $\mathbb{R}^n$. The definition and properties of the fundamental group can be found in \cite{munkres2000}. All simplices are simply-connected, so the fundamental group of an individual simplex is trivial. 

Now let $\mathit{X}$ be a set of positive integers, and let $\mathit{X}^* = \mathit{X}\setminus\{1\}$. We will define two simplicial complexes from $X^*$. 
\begin{mydef}
The $\textbf{common divisor simplicial complex}$ of $X$, which we will denote $\mathcal{G}(X)$, has as its vertices the elements of $\mathit{X}^*$. Given a set of $n+1$ vertices $V\subseteq X^*$, we form an $n$-simplex out of $V$ if $\gcd(V)>1$. 
\end{mydef}
\begin{mydef}
The $\textbf{prime divisor simplicial complex}$ of $X$, which we will denote $\mathcal{D}(X)$, has as its vertices the set of primes dividing at least one member of $X^*$. If $P$ is the set of primes dividing some member of $X^*$, and $Q\subseteq P$ consists of $n+1$ primes, we form an $n$-simplex out of $Q$ if there exists an element $v$ of $\mathit{X}$ such that every element of $Q$ divides $v$. 
\end{mydef}

For an example of the two simplicial complexes, see Section \ref{example}.

Because every vertex in $\mathcal{G}(X)$ has a corresponding integer from $X$, we will use the phrase ``divides a vertex'' to mean ``divides the integer corresponding to the vertex.''

Looking at examples of both simplicial complexes, we see that there is a similarity of structure between $\mathcal{G}(X)$ and $\mathcal{D}(X)$. The two are not the same, but for any given set of integers $X$ they always seem to have the same number of loops. We will show that the two simplicial complexes have isomorphic fundamental groups. To do this, we will define a map from $\mathcal{G}(X)$ to $\mathcal{D}(X)$ that induces a map on the corresponding fundamental groups. 

\section{A Correspondence between $\mathcal{G}(X)$ and $\mathcal{D}(X)$}\label{correspondence}

We define a function that we will use to construct a map between $\mathcal{G}(X)$ and $\mathcal{D}(X)$. Let $v$ be an element of $\mathit{X}$. By the Fundamental Theorem of Arithmetic, $v$ has a unique prime factorization. Define $\pi(v)$ to be the set of distinct primes $p$ such that $p$ divides $v$, and note that $|\pi(v)|$ is the number of distinct primes in the prime factorization of $v$, not including multiplicity. That is, if $p^2|v$, this term only contributes 1 to $|\pi(v)|$. For example, let $v:=2^2\cdot5$. Then $|\pi(v)|=2$, one for the two 2 factors and one for the 5 factor. Given a set $\sigma$ of integers, we define $\pi(\sigma)$ to be $\pi(\gcd(\sigma))$.

We are ready to define our map, which we will denote by $\eta$. 

\begin{mydef}
Given $\mathit{X}$ a set of integers, let the domain of $\eta$ be the set of simplices in $\mathcal{G}(X)$, let the codomain be the set of simplices in $\mathcal{D}(X)$, and let $\sigma=\{v_1,v_2,\ldots,v_{n+1}\}$ be an $n$-simplex in $\mathcal{G}(X)$. Define $\eta(\sigma)$ to be the simplex in $\mathcal{D}(X)$ with the members of $\pi(\sigma)$ as vertices.  
\end{mydef}

Note that for all $\sigma\in\mathcal{G}(X)$, $\eta(\sigma)$ has dimension $|\pi(\sigma)|-1$.

The map $\eta$ is not an injection for all sets. For example, let $\mathit{X}:=\{2, 4, 8\}$. Then $\mathcal{D}(X)$ consists of a single point, $2$, onto which all three vertices, all three 1-simplices, and the 2-simplex of $\mathcal{G}(X)$ are mapped. It is also possible to have a simplex in $\mathcal{G}(X)$ map onto a simplex in $\mathcal{D}(X)$, and another simplex in $\mathcal{G}(X)$ map onto a face of the same simplex in $\mathcal{D}(X)$. 

The map $\eta$ is not a surjection for all sets, but it does have some nice properties similar to a surjection. Let $p$ be a vertex in $\mathcal{D}(X)$. Then there exists $v\in\mathit{X}$ such that $p|v$. We see that $\eta(v)$ is the $(|\pi(v)|-1)$-simplex with $p$ as one of its vertices in $\mathcal{D}(X)$. This means that not every vertex is the image of a simplex, but every vertex is a face of the image of a simplex. Likewise let $\sigma$ be any $n$-simplex in $\mathcal{D}(X)$. This means that there exists $v\in\mathit{X}$ and $n+1$ primes $p_1,p_2,\ldots,p_{n+1}$ such that $p_1p_2\ldots p_{n+1}|v$, from which we conclude that $\eta(v)$ is either $\sigma$ or a simplex with $\sigma$ as one of its faces. We see that for all primes $p\in P$ there exists some simplex $\sigma\in\mathcal{D}(X)$ with $p$ a vertex of $\sigma$ such that $\sigma = \eta(\tau)$, where $\tau$ is a simplex in $\mathcal{G}(X)$.

For example, let $X=\{30\}$, so the set of primes $P$ which divide elements of $X$ is $\{2,3,5\}$. Now there is no element of $\mathcal{G}(X)$ which maps to the vertex $2$ in $\mathcal{D}(X)$, but $\eta(30)$ is the $2$-simplex between $2$, $3$, and $5$ which contains the vertex $2$. 

The following result appears as Lemma 3.1 in \cite{lewis2008overview}.

\begin{theorem}\label{connectedness}
For all sets of integers $X$, $\mathcal{G}(X)$ is connected if and only $\mathcal{D}(X)$ is connected. 
\end{theorem}

From now on we will assume that our simplicial complexes are connected.

\begin{mydef}
The \textbf{2-skeleton} of a simplicial complex $\mathcal{K}$, denoted $\mathcal{K}^2$, is the pair $(\mathcal{V}, \mathcal{S}')$, where $\mathcal{S}'=\{\sigma\in\mathcal{S}:\lvert\sigma\rvert\leq 3\}$.
\end{mydef}
The 2-skeleton $\mathcal{K}^2$ is simply the union of all the simplices of $\mathcal{K}$ of dimension 2 or fewer. Given a simplicial complex $\mathcal{K}$, $\mathcal{K}^2$ is a new simplicial complex similar to $\mathcal{K}$ but with all 3-simplices and higher removed. The 2-dimensional faces of the higher dimensional simplices, which were removed, remain. 

Using the concept of a 2-skeleton, we will discuss some nice properties of the fundamental group of a simplicial complex. We will use the following result, which appears as Proposition 1.26 in \cite{hatcher2002algebraic}.

\begin{theorem}\label{skel}
The fundamental group of a simplicial complex depends only on its 2-skeleton. 
\end{theorem} 

\begin{mydef}
If $\mathcal{K}$ is a connected simplicial complex, an \textbf{edge-path} is a chain of vertices connected by edges in $\mathcal{K}$. An \textbf{edge-loop} is an edge-path starting and ending at the same vertex. 
\end{mydef}

Next, we will define the notation for an edge-path. Let $v_1, v_2, \ldots, v_n$ be a sequence of vertices with $v_i$ and $v_{i+1}$ connected by edges for $i$ from $1$ to $n-1$. For convenience when describing edges, let the edge between $v_i$ and $v_{i+1}$ be denoted by $\langle v_i,v_{i+1}\rangle$. Let the path along these vertices be denoted as $\langle v_1, v_2\rangle\langle v_2, v_3\rangle, \ldots, \langle v_{n-1},v_n\rangle$. Note that using  this notation, the edge-path is an edge-loop if and only if $v_1 = v_n$. 

\begin{theorem}
Every loop in a simplicial complex is homotopic to an edge-loop.
\end{theorem}
\begin{proof}
By Theorem \ref{skel}, we know that the fundamental group of a simplicial complex only depends on its 2-skeleton, so every loop in the complex is homotopic to a loop in the 2-skeleton. Now consider a loop that passes through a 2-simplex. Because simplices are simply-connected, this loop can be continuously deformed to the edge of the 2-simplex. 
\end{proof}
Note that the fundamental group of the simplicial complex is not isomorphic to the fundamental group of the 1-skeleton. To see this, consider a simplicial complex consisting only of a 2-simplex. The fundamental group of the 2-simplex is the trivial group, whereas the fundamental group of the 1-skeleton of the 2-simplex is $\mathbb{Z}$. 

\begin{mydef}
Let $v_1$, $v_2$, and $v_3$ be three not necessarily distinct vertices in a simplicial complex; let $\mathcal{P}_1$ be a path ending at $v_1$ and let $\mathcal{P}_2$ be a path beginning at $v_2$. We define the path $\mathcal{P}_1\langle v_1, v_3\rangle\langle v_3, v_2\rangle \mathcal{P}_2$ to be \textbf{simply equivalent} to the path $\mathcal{P}_1\langle v_1, v_2\rangle \mathcal{P}_2$ if and only if there exists a simplex containing $v_1$, $v_2$, and $v_3$.

\end{mydef}
Given two edge-paths $\mathcal{P}_1$ and $\mathcal{P}_2$, we use the notation $\mathcal{P}_1\sim \mathcal{P}_2$ to mean that $\mathcal{P}_1$ is simply equivalent to $\mathcal{P}_2$. 

Note that the idea of a simple equivalence can take three forms, depending on which of $v_1$, $v_2$, and $v_3$ are distinct. If the three vertices are distinct, the simple equivalence takes the path $P_1\langle v_1, v_3\rangle\langle v_3, v_2\rangle P_2$ to the path $P_1\langle v_1, v_2\rangle P_2$ if and only if there exists a 2-simplex between the vertices $v_1$, $v_2$, and $v_3$. 

If $v_1=v_2$, edge-paths of the form $P_1\langle v_1, v_3\rangle\langle v_3, v_1\rangle P_2$ and $P_1 P_2$ are always simply equivalent, because by the definition of an edge-path there is a 1-simplex between $v_1$ and $v_3$.

Now we consider three cases where the vertices are not distinct. If $v_1 = v_3$, $v_2 = v_3$, or $v_1=v_2=v_3$, the two edge-paths that are simply equivalent will look the same -- the only difference is that repeated vertices in the path are removed.

Given two edge-loops, if one can be obtained from the other by a finite number of simple equivalences we say the loops are homotopic.  It is a fact homotopy forms an equivalence relation on the edge-loops of a simplicial complex.  With these equivalence classes we can define a group of equivalence classes of edge-loops on the 2-skeleton of the simplicial complex. By Theorem \ref{skel}, this group of equivalence classes of edge-loops in the 2-skeleton is isomorphic to the fundamental group of the simplicial complex. 

We will now show that $\eta$ induces a map on edge-paths. But first we need a theorem regarding a very interesting subset reversal property of $\eta$. 

\begin{theorem}(Subset Reversal)\label{face}
Given two simplices $\sigma_1$ and $\sigma_2$ of $\mathcal{G}(X)$, if $\sigma_1$ is a face of $\sigma_2$, then $\eta(\sigma_2)$ is a face of $\eta(\sigma_1)$. 
\end{theorem}
\begin{proof}
Because $\sigma_1$ is a face of $\sigma_2$, its set of vertices are a subset of the set of vertices of $\sigma_2$. Therefore the set of primes of the greatest common divisor of the vertices of $\sigma_2$ are a subset of the set of primes of the greatest common divisor of the vertices of $\sigma_1$. This means that the dimension of $\eta(\sigma_2)$ is less than or equal to the dimension of $\eta(\sigma_1)$, and because the vertices in $\eta(\sigma_2)$ are also vertices of $\eta(\sigma_1)$ we have that $\eta(\sigma_2)$ is a face of $\eta(\sigma_1)$.
\end{proof}

\section{The Induced Map}\label{inducedmap}

We will show that the map $\eta$ induces a map $\eta_*$ from edge-paths in $\mathcal{G}(X)$ to edge-paths in $\mathcal{D}(X)$. For edge-paths, we are only concerned with $\eta$'s effect on 0 and 1-simplices. Because an edge-path consists of an alternating sequence of vertices and edges, we can look at the effect that $\eta$ has on vertices and edges. By Theorem \ref{face}, the image of an edge is a face of the images of the vertices on either side of the edge, so we can simply look at the image of the edges in the path.
\begin{mydef}\label{inducedmapdef}
Let $\mathcal{L}=\langle v_1,v_2\rangle\ldots\langle v_{n-1},v_n\rangle$ be an edge-path in $\mathcal{G}(X)$. Fix a point $\omega_i\in\eta(\langle v_i,v_{i+1}\rangle)$ for each $1\leq i\leq n-1$, and fix another point $\omega_n\in\eta(v_n)$. Define $\eta_*(\mathcal{L}) = \langle \omega_1,\omega_2\rangle\ldots\langle \omega_{n-1},\omega_n\rangle$. 
\end{mydef}
We know that edges exist between each successive $\omega_i$ and $\omega_{i+1}$ in $\mathcal{D}(X)$ because $\omega_i$ and $\omega_{i+1}$ both divide $v_{i+1}$. Note that the paths given for all possible choices for $\omega_i$ with the same endpoints are homotopic because every possible choice for $\omega_i$ is necessarily inside the simplex $\eta(\langle v_i, v_{i+1}\rangle)$, and simplices are simply-connected. Because edge-loops are edge-paths with $v_1=v_n$, $\eta_*$ can be applied to them as well if we choose $\omega_n$ to be the same point as $\omega_1$. 

\begin{theorem}
The induced map $\eta_*$ is a well-defined map on edge-paths. 
\end{theorem}
\begin{proof}
Let $\mathcal{L}_0$ and $\mathcal{L}_1$ be two homotopic paths in $\mathcal{G}(X)$. We will show that their images under $\eta_*$ are homotopic in $\mathcal{D}(X)$. For $\mathcal{L}_0$ and $\mathcal{L}_1$ to be homotopic in the domain means that there is a sequence of simple equivalences from $\mathcal{L}_0$ to $\mathcal{L}_1$.  Assume $\mathcal{L}_0$ and $\mathcal{L}_1$ differ by one simple equivalence. Let $\mathcal{L}_0$  contain the edge-path segment $\mathcal{P}_0=\langle v_a,v_1\rangle\langle v_1,v_2\rangle\langle v_2,v_b\rangle$ and $\mathcal{L}_1$ contain the edge-path segment $\mathcal{P}_1=\langle v_a,v_1\rangle\langle v_1,v_3\rangle\langle v_3,v_2\rangle\langle v_2,v_b\rangle$, where $v_1$, $v_2$, and $v_3$ are not necessarily distinct. For $\mathcal{L}_0$ and $\mathcal{L}_1$ to differ by one simple equivalence, there must exist a simplex with $v_1$, $v_2$, and $v_3$ as vertices. Because this simplex exists, $v_1$, $v_2$, and $v_3$ must share a prime divisor which we will call $\omega_{123}$. Now we will apply Definition \ref{inducedmapdef} to $\mathcal{L}_0$ and $\mathcal{L}_1$ to find $\eta_*(\mathcal{L}_0)$ and $\eta_*(\mathcal{L}_1)$. For every edge in the two loops, we know that there exists a prime dividing the vertices on either side of the edge. Let $\omega_{a1}$ divide $v_a$ and $v_1$, $\omega_{12}$ divide $v_1$ and $v_2$, $\omega_{13}$ divide $v_1$ and $v_3$, $\omega_{23}$ divide $v_2$ and $v_3$, and $\omega_{2b}$ divide $v_2$ and $v_b$. We see that $\eta_*(\mathcal{L}_0)$ passes through the simplices $\eta(\langle v_a,v_1\rangle)$, $\eta(\langle v_1,v_2\rangle)$, and $\eta(\langle v_2,v_b\rangle)$. Because simplices are simply-connected any path through these simplices can be continuously deformed to the path $\langle \omega_{a1},\omega_{12}\rangle\langle \omega_{12},\omega_{2b}\rangle$ in $\mathcal{D}(X)$. Likewise $\eta_*(\mathcal{L}_1)$ passes through $\eta(\langle v_a,v_1\rangle)$, $\eta(\langle v_1,v_3\rangle)$, $\eta(\langle v_3,v_2\rangle)$, and $\eta(\langle v_2,v_b\rangle)$, and this segment of $\mathcal{L}_1$ can be continuously deformed to the path $\langle \omega_{a1},\omega_{13}\rangle\langle \omega_{13},\omega_{23}\rangle\langle \omega_{23},\omega_{2b}\rangle$. That is, $\eta_*(\mathcal{P}_0)=\langle \omega_{a1},\omega_{12}\rangle\langle \omega_{12},\omega_{2b}\rangle$ and $\eta_*(\mathcal{P}_1)=\langle \omega_{a1},\omega_{13}\rangle\langle \omega_{13},\omega_{23}\rangle\langle \omega_{23},\omega_{2b}\rangle$.

Our goal now is to show that because $\mathcal{L}_0$ and $\mathcal{L}_1$ are homotopic, the images of $\mathcal{L}_0$ and $\mathcal{L}_1$ under $\eta_*$ are homotopic. Since $\mathcal{L}_0$ and $\mathcal{L}_1$ agree everywhere except for $\mathcal{P}_0$ and $\mathcal{P}_1$, it suffices to show that the images of these two paths under $\eta_*$ are homotopic. Now because $\omega_{a1}$, $\omega_{12}$, and $\omega_{13}$ all divide $v_1$, these primes form a simplex in $\mathcal{D}(X)$. Likewise since $\omega_{12}$, $\omega_{23}$, and $\omega_{2b}$ all divide $v_2$, these primes also form a simplex. In the same way because $\omega_{13}$, $\omega_{12}$, and $\omega_{123}$ all divide $v_1$, $\omega_{23}$, $\omega_{12}$, and $\omega_{123}$ all divide $v_2$, and $\omega_{13}$, $\omega_{23}$, and $\omega_{123}$ all divide $v_3$, these three groups of primes form three simplices in $\mathcal{D}(X)$. Because we have all of these simplices, we see that 
\begin{align*}
\langle\omega_{a1},\omega_{13}\rangle\langle\omega_{13},\omega_{23}\rangle\langle\omega_{23},\omega_{2b}\rangle &\sim\langle\omega_{a1},\omega_{13}\rangle\langle \omega_{13},\omega_{123}\rangle\langle \omega_{123},\omega_{23}\rangle\langle\omega_{23},\omega_{2b}\rangle\\
&\sim\langle\omega_{a1},\omega_{13}\rangle\langle\omega_{13},\omega_{12}\rangle\langle \omega_{12},\omega_{123}\rangle\langle \omega_{123},\omega_{23}\rangle\langle\omega_{23},\omega_{2b}\rangle\\
&\sim\langle\omega_{a1},\omega_{13}\rangle\langle\omega_{13},\omega_{12}\rangle\langle \omega_{12},\omega_{23}\rangle\langle\omega_{23},\omega_{2b}\rangle\\
&\sim\langle\omega_{a1},\omega_{12}\rangle\langle \omega_{12},\omega_{23}\rangle\langle\omega_{23},\omega_{2b}\rangle\\
&\sim\langle\omega_{a1},\omega_{12}\rangle\langle \omega_{12},\omega_{2b}\rangle.
\end{align*}
Thus $\eta_*(\mathcal{L}_0)$ is homotopic to $\eta_*(\mathcal{L}_1)$. 

Likewise if $\mathcal{L}_0$ and $\mathcal{L}_1$ differ by a finite number of successive simple equivalences, we can find a chain of intermediary paths from $\mathcal{L}_0$ and $\mathcal{L}_1$, each differing by a simple equivalence, with each path in the chain homotopic to the last. Since each image in the chain is homotopic to the image of the last, we know that $\eta_*(\mathcal{L}_0)$ and $\eta_*(\mathcal{L}_1)$ are homotopic by transitivity. Therefore the induced map $\eta_*$ is well-defined. 
\end{proof}

Having established that $\eta_*$ is well-defined, we now explicitly define the inverse map $\eta_*^{-1}$ and show that it satisfies the properties of the inverse of $\eta_*$. That is, given an edge-loop $\mathcal{L}$ in $\mathcal{G}(X)$, we will show that the path $\eta_*^{-1}(\eta_*(\mathcal{L}))$ is homotopic to $\mathcal{L}$ in $\mathcal{G}(X)$.  Similarly, if $\mathcal{L}$ is an edge-loop in $\mathcal{D}(X)$, we show that $\eta_*(\eta_*^{-1}(\mathcal{L}))$ is homotopic to $\mathcal{L}$ in $\mathcal{D}(X)$. 

\begin{mydef}\label{inverseinducedmap}
Let $\mathcal{L}=\langle p_1,p_2\rangle\ldots\langle p_{n-1},p_n\rangle$ be an edge-path in $\mathcal{D}(X)$. For every point $p_i\in\mathcal{L}$ for $1\leq i\leq n$, there exists a simplex $\sigma_i\in\mathcal{G}(X)$ such that $p_i$ divides $v_j$ for all $v_j\in\sigma_i$. Fix a point $\alpha_1\in\sigma_1$, and for each $2\leq i\leq n$, fix a point $\alpha_i\in\sigma_{i-1}\cap\sigma_{i}$. Define $\eta_*^{-1}(\mathcal{L}) = \langle\alpha_1,\alpha_2 \rangle\ldots\langle \alpha_{n-1},\alpha_n\rangle$. 
\end{mydef}

Note that the intersection between each successive $\sigma_{i-1}$ and $\sigma_i$ is non-empty because $p_{i-1}$ and $p_i$ are connected by an edge in $\mathcal{D}(X)$. We know that edges exist between each successive $\alpha_{i-1}$ and $\alpha_i$ in $\mathcal{G}(X)$ because $\alpha_{i-1}$ and $\alpha_i$ are both divisible by $p_i$. 

Again note that the paths given for all possible choices for $\alpha_i$ with the same endpoints are homotopic because every possible choice for $\alpha_i$ is necessarily inside the simplex $\sigma_{i-1}\cap\sigma_i$, and simplices are simply-connected. Because edge-loops are edge-paths with $p_1=p_n$, $\eta_*^{-1}$ can be applied to them as well if we choose $\alpha_n$ to be the same point as $\alpha_1$. 

\begin{theorem}\label{inverse}
Given an edge-loop $\mathcal{L}=\langle v_1,v_2\rangle\ldots\langle v_{n-1},v_n\rangle\in\mathcal{G}(X)$, $\eta_*^{-1}(\eta_*(\mathcal{L}))$ is homotopic to $\mathcal{L}$.
\end{theorem}
\begin{proof}
Given an edge-loop $\mathcal{L}=\langle v_1,v_2\rangle\ldots\langle v_{n-1},v_n\rangle\in\mathcal{G}(X)$, we will show that $\eta_*^{-1}(\eta_*(\mathcal{L}))$ is homotopic to $\mathcal{L}$. Using the definition of $\eta_*$, we construct the path $\eta_*(\mathcal{L}) = \langle \omega_1,\omega_2\rangle\ldots\langle \omega_{n-1},\omega_n\rangle$ so that $\omega_i$ divides $v_i$ and $v_{i+1}$ for $1\leq i\leq n-1$. In the same way, using the definition of $\eta_*^{-1}$, we construct the path $\eta_*^{-1}(\eta_*(\mathcal{L})) = \langle \alpha_1,\alpha_2\rangle\ldots\langle \alpha_{n-1},\alpha_n\rangle$, choosing $\alpha_1=v_1$ and $\alpha_i$ divisible by $\omega_{i-1}$ and $\omega_i$ for $2 \leq i\leq n$. Now since $\omega_i$ divides each of $v_i$, $v_{i+1}$, $\alpha_i$, and $\alpha_{i+1}$ for all $i$, these vertices are members of the same simplex in $\mathcal{G}(X)$; the simplex with all of its vertices divisible by $\omega_i$. Therefore $\eta_*^{-1}(\eta_*(\mathcal{L}))$ can be continuously deformed to $\mathcal{L}$ and they have the same base point, so the two paths are homotopic. 
\end{proof}
\begin{theorem}\label{inverse2}
Given an edge-loop $\mathcal{L}=\langle p_1,p_2\rangle\ldots\langle p_{n-1},p_n\rangle\in\mathcal{D}(X)$, $\eta_*(\eta_*^{-1}(\mathcal{L}))$ is homotopic to $\mathcal{L}$.
\end{theorem}
\begin{proof}
Given an edge-loop $\mathcal{L}=\langle p_1,p_2\rangle\ldots\langle p_{n-1},p_n\rangle\in\mathcal{D}(X)$, we will show that $\eta_*(\eta_*^{-1}(\mathcal{L}))$ is homotopic to $\mathcal{L}$. Using the definition of $\eta_*^{-1}$, we construct the path $\eta_*^{-1}(\mathcal{L}) = \langle \alpha_1,\alpha_2\rangle\ldots\langle \alpha_{n-1},\alpha_n\rangle$ so that $\alpha_i$ is divisible by $p_i$ and $p_{i+1}$ for $2 \leq i\leq n$. In the same way, using the definition of $\eta_*$, we construct the path $\eta_*(\eta_*^{-1}(\mathcal{L})) = \langle \omega_1,\omega_2\rangle\ldots\langle \omega_{n-1},\omega_n\rangle$, choosing $\omega_1=p_1$ and $\omega_i$ dividing $\alpha_i$ and $\alpha_{i+1}$ for $1\leq i\leq n-1$. Now since $p_i$, $p_{i-1}$, $\omega_i$, and $\omega_{i-1}$ all divide $\alpha_i$ for all $i$, these primes are members of the same simplex in $\mathcal{D}(X)$; the simplex where all vertices divide $\omega_i$. Therefore $\eta_*(\eta_*^{-1}(\mathcal{L}))$ can be continuously deformed to $\mathcal{L}$ and they have the same base point, so the two paths are homotopic. 
\end{proof}

We will now show that $\eta_*$ is a homomorphism on $\pi_1(\mathcal{G}(X))$ and $\pi_1(\mathcal{D}(X))$; that is, we will show that $\eta_*$ is operation preserving. 

\begin{theorem}\label{homomorphism}
The map $\eta_*$ is a homomorphism between $\pi_1(\mathcal{G}(X))$ and $\pi_1(\mathcal{D}(X))$. 
\end{theorem}
\begin{proof}
Let $\mathcal{L}_0$ and $\mathcal{L}_1$ be paths in $\mathcal{G}(X)$. Recall that the operation of the fundamental group is concatenation of paths, and that paths are continuous functions from $[0,1]$ to the topological space; that is, 
\begin{equation*}
\mathcal{L}_0\circ \mathcal{L}_1 = \left\{\begin{array}{lr}
       \mathcal{L}_0(2t) & : t\in [0,\frac{1}{2}]\\
       \mathcal{L}_1(2t-1) & : t \in [\frac{1}{2},1]\end{array}
   \right. .
   \end{equation*}
   Now 
   \begin{equation*}
   \eta_*(\mathcal{L}_0\circ \mathcal{L}_1)=\left\{\begin{array}{lr}
       \eta_*(\mathcal{L}_0(2t)) & : t\in [0,\frac{1}{2}]\\
       \eta_*(\mathcal{L}_1(2t-1)) & : t \in [\frac{1}{2},1]\end{array}\right. = \eta_*(\mathcal{L}_0)\circ\eta_*(\mathcal{L}_1),
       \end{equation*}
       so $\eta_*$ is a group homomorphism between $\pi_1(\mathcal{G}(X))$ and $\pi_1(\mathcal{D}(X))$. 
\end{proof}

\begin{theorem}\label{thmsurjection}
The map $\eta_*$, as a homomorphism, is a surjection. 
\end{theorem}
\begin{proof}
Let $\mathcal{L}$ be an edge-loop in $\mathcal{D}(X)$. We will show that there exists an edge-loop in $\mathcal{G}(X)$ that maps to an edge-loop homotopic to $\mathcal{L}$ under $\eta_*$. Consider the edge-loop $\eta_*^{-1}(\mathcal{L})$. Based on Definition \ref{inverseinducedmap}, this loop can be constructed no matter the choice of $\mathcal{L}$. By Theorem \ref{inverse2}, $\eta_*(\eta_*^{-1}(\mathcal{L}))$ is homotopic to $\mathcal{L}$. Thus $\eta_*$ is a surjection. 
\end{proof}

\begin{theorem}\label{inversewelldefined}
The inverse map $\eta_*^{-1}$ is a well-defined map on edge-paths. 
\end{theorem}

\begin{proof}
Let $\mathcal{L}_0$ and $\mathcal{L}_1$ be two homotopic paths in $\mathcal{D}(X)$. For $\mathcal{L}_0$ and $\mathcal{L}_1$ to be homotopic in $\mathcal{D}(X)$ means that there is a sequence of simple equivalences from $\mathcal{L}_0$ to $\mathcal{L}_1$.  Assume $\mathcal{L}_0$ and $\mathcal{L}_1$ differ by one simple equivalence. Let $\mathcal{L}_0$  contain the edge-path segment $\mathcal{P}_0=\langle p_a,p_1\rangle\langle p_1,p_2\rangle\langle p_2,p_b\rangle$ and $\mathcal{L}_1$ contain the edge-path segment $\mathcal{P}_1=\langle p_a,p_1\rangle\langle p_1,p_3\rangle\langle p_3,p_2\rangle\langle p_2,p_b\rangle$, where $p_1$, $p_2$, and $p_3$ are not necessarily distinct. For $\mathcal{L}_0$ and $\mathcal{L}_1$ to differ by one simple equivalence, there must exist a simplex with $p_1$, $p_2$, and $p_3$ as vertices. Because this simplex exists, $p_1$, $p_2$, and $p_3$ must divide a vertex in $\mathcal{G}(X)$ which we will call $\alpha_{123}$. Now we will apply Definition \ref{inverseinducedmap} to $\mathcal{L}_0$ and $\mathcal{L}_1$ to find $\eta_*^{-1}(\mathcal{L}_0)$ and $\eta_*^{-1}(\mathcal{L}_1)$. For every edge in the two loops, we know that there exists a vertex in $\mathcal{G}(X)$ which is divisible by the primes on either side of the edge. Let $\alpha_{a1}$ be divisible by $p_a$ and $p_1$, $\alpha_{12}$ be divisible by $p_1$ and $p_2$, $\alpha_{13}$ be divisible by $p_1$ and $p_3$, $\alpha_{23}$ be divisible by $p_2$ and $p_3$, and $\alpha_{2b}$ be divisible by $p_2$ and $p_b$. Let $\sigma_a$ be the simplex in $\mathcal{G}(X)$ in which every vertex is divisible by $p_a$; define $\sigma_1, \sigma_2, \sigma_3$, and $\sigma_b$ analogously.  We see that $\eta_*^{-1}(\mathcal{L}_0)$ passes through the simplices $\sigma_a$, $\sigma_1$, $\sigma_2$, and $\sigma_b$. Because simplices are simply-connected any path through these simplices can be continuously deformed to the path $\langle \alpha_{a1},\alpha_{12}\rangle\langle \alpha_{12},\alpha_{2b}\rangle$ in $\mathcal{G}(X)$. Likewise $\eta_*^{-1}(\mathcal{L}_1)$ passes through the simplices $\sigma_a$, $\sigma_1$, $\sigma_3$, $\sigma_2$, and $\sigma_b$, and this segment of $\mathcal{L}_1$ can be continuously deformed to the path $\langle \alpha_{a1},\alpha_{13}\rangle\langle \alpha_{13},\alpha_{23}\rangle\langle \alpha_{23},\alpha_{2b}\rangle$. That is, $\eta_*^{-1}(\mathcal{P}_0)=\langle \alpha_{a1},\alpha_{12}\rangle\langle \alpha_{12},\alpha_{2b}\rangle$ and $\eta_*^{-1}(\mathcal{P}_1)=\langle \alpha_{a1},\alpha_{13}\rangle\langle \alpha_{13},\alpha_{23}\rangle\langle \alpha_{23},\alpha_{2b}\rangle$.

Our goal now is to show that because $\mathcal{L}_0$ and $\mathcal{L}_1$ are homotopic, the inverse images of $\mathcal{L}_0$ and $\mathcal{L}_1$ under $\eta_*$ are homotopic. Since $\mathcal{L}_0$ and $\mathcal{L}_1$ agree everywhere except for $\mathcal{P}_0$ and $\mathcal{P}_1$, it suffices to show that the inverse images of these two paths under $\eta_*$ are homotopic. Now because $\alpha_{a1}$, $\alpha_{12}$, and $\alpha_{13}$ are all divisible by $p_1$, these vertices form a simplex in $\mathcal{G}(X)$. Likewise since $\alpha_{12}$, $\alpha_{23}$, and $\alpha_{2b}$ are all divisible by $p_2$, these vertices also form a simplex. In the same way because $\alpha_{13}$, $\alpha_{12}$, and $\alpha_{123}$ are all divisible by $p_1$, $\alpha_{23}$, $\alpha_{12}$, and $\alpha_{123}$ are all divisible by $p_2$, and $\alpha_{13}$, $\alpha_{23}$, and $\alpha_{123}$ are all divisible by $p_3$, these three groups of vertices form three simplices in $\mathcal{G}(X)$. Because we have all of these simplices, we see that 
\begin{align*}
\langle\alpha_{a1},\alpha_{13}\rangle\langle\alpha_{13},\alpha_{23}\rangle\langle\alpha_{23},\alpha_{2b}\rangle &\sim\langle\alpha_{a1},\alpha_{13}\rangle\langle \alpha_{13},\alpha_{123}\rangle\langle \alpha_{123},\alpha_{23}\rangle\langle\alpha_{23},\alpha_{2b}\rangle\\
&\sim\langle\alpha_{a1},\alpha_{13}\rangle\langle\alpha_{13},\alpha_{12}\rangle\langle \alpha_{12},\alpha_{123}\rangle\langle \alpha_{123},\alpha_{23}\rangle\langle\alpha_{23},\alpha_{2b}\rangle\\
&\sim\langle\alpha_{a1},\alpha_{13}\rangle\langle\alpha_{13},\alpha_{12}\rangle\langle \alpha_{12},\alpha_{23}\rangle\langle\alpha_{23},\alpha_{2b}\rangle\\
&\sim\langle\alpha_{a1},\alpha_{12}\rangle\langle \alpha_{12},\alpha_{23}\rangle\langle\alpha_{23},\alpha_{2b}\rangle\\
&\sim\langle\alpha_{a1},\alpha_{12}\rangle\langle \alpha_{12},\alpha_{2b}\rangle.
\end{align*}
Thus $\eta_*^{-1}(\mathcal{L}_0)$ is homotopic to $\eta_*^{-1}(\mathcal{L}_1)$. 

Likewise if $\mathcal{L}_0$ and $\mathcal{L}_1$ differ by a finite number of successive simple equivalences, we can find a chain of intermediary paths from $\mathcal{L}_0$ and $\mathcal{L}_1$, each differing by a simple equivalence, with each path in the chain homotopic to the last. Since each image in the chain is homotopic to the image of the last, we know that $\eta_*^{-1}(\mathcal{L}_0)$ and $\eta_*^{-1}(\mathcal{L}_1)$ are homotopic by transitivity. Therefore the inverse map $\eta_*^{-1}$ is well-defined. 
\end{proof}

\begin{corollary}\label{onetoone}
The map $\eta_*$ is an injection. 
\end{corollary}

\begin{proof}

It suffices to show that the kernel of $\eta_*$ is trivial by Theorem 10.2.9 in \cite{gallian2009}. We will show that if a loop in $\mathcal{D}(X)$ is homotopic to the trivial loop, then the inverse image of that loop under $\eta_*$ is homotopic to the trivial loop in $\mathcal{G}(X)$. 

Let $\mathcal{L}\in\ker(\eta_*)$ with base point $b$. We will show that $\mathcal{L}$ is homotopic to the trivial loop.  Let $\omega$ be the basepoint of $\eta_*(\mathcal{L})$.  Because $\mathcal{L}\in\ker(\eta_*)$, we know that $\eta_*(\mathcal{L})$ is homotopic to $\langle \omega, \omega \rangle$ in $\mathcal{D}(X)$.  By Theorem \ref{inversewelldefined}, $\eta_*^{-1}$ is well-defined, and therefore $\eta_*^{-1}(\eta_*(\mathcal{L}))$ is homotopic to $\eta_{*}^{-1}(\langle \omega, \omega \rangle)$.  To preserve basepoints, we take $\eta_{*}^{-1}(\langle \omega, \omega \rangle) = \langle b,b \rangle$.  Applying Theorem \ref{inverse}, we see that $\mathcal{L}$ is homotopic to $\eta_{*}^{-1}(\eta_*(\mathcal{L}))$, and we conclude that $\mathcal{L}$ is homotopic to $\langle b, b \rangle$.  As $b$ is a single point, the loop $\langle b , b \rangle$ is trivial.  Since $\mathcal{L} \in \ker(\eta_*)$ was arbitrary, we have that $\eta_*$ is an injection. 
\end{proof}

\begin{corollary}\label{bijection}
The map $\eta_*$ is an isomorphism from $\pi_1(\mathcal{G}(X))$ to $\pi_1(\mathcal{D}(X))$. 
\end{corollary}
\begin{proof}
By Theorem \ref{homomorphism}, The map $\eta_*$ is a group homomorphism from $\pi_1(\mathcal{G}(X))$ to $\pi_1(\mathcal{D}(X))$.  Combining Theorem \ref{thmsurjection} and Corollary \ref{onetoone}, we have that $\eta_*$ is both one-to-one and onto.  Thus $\eta_*$ is an isomorphism between $\pi_1(\mathcal{G}(X))$ and $\pi_1(\mathcal{D}(X))$. 
\end{proof}

Thus for any set of integers, the fundamental groups of the common and the prime divisor simplicial complexes are isomorphic. As a corollary to this, we have the following result. 

\begin{corollary}
The first homology groups of $\mathcal{G}(X)$ and $\mathcal{D}(X)$ are isomorphic. 
\end{corollary}
\begin{proof}
As the first homology group is simply the abelianization of the fundamental group, the result holds. 
\end{proof}

\section{Applications}\label{applications}
 
We now apply the more general results of the previous sections to finite group theory.  In all that follows, we assume that $G$ is a finite group, writing $\Irr(G)$ for the irreducible characters of $G$ and $\cd(G) = \{\chi(1) \, \, | \, \, \chi \in \Irr(G)\}$.  We abbreviate $\Gamma(G)$ for the common divisor graph of $\cd(G)\setminus \{1\}$ and $\Delta(G)$ for the prime divisor graph of $\cd(G)$.  If $x$ is an integer, recall that $\pi(x)$ denotes the set of prime divisors of $x$.
 
As alluded to in the introduction, the research that begins with the graph $\Gamma(G)$ or $\Delta(G)$ and derives group theoretic properties about the corresponding group has recently been expanded.  It seems that much is to be gained by studying the more rich structure of the \emph{common divisor simplicial complex}, denoted $\mathcal{G}(G)$, or the \emph{prime divisor simplicial complex}, hereafter $\mathcal{D}(G)$, of the set $\cd(G)$.  This suspicion that is mentioned in \cite{lewis2008overview} is confirmed, at least for the common divisor simplicial complex, in \cite{jensen2015character}.  The author of \cite{jensen2015character} chooses only to work with $\mathcal{G}(G)$, and all results are stated for $\mathcal{G}(G)$ only.  Combining the results of \cite{jensen2015character} with those of this paper, it is now possible to restate those results in terms of $\mathcal{D}(G)$.
 
 The first result from \cite{jensen2015character} that carries over is the following.
 \begin{lemma}
 Suppose $G$ is a finite group and $K \lhd G$ is maximal so that $G/K$ is nonabelian.  If $G/K$ is a $p$-group for some prime $p$ then $\mathcal{D}(G)$ is connected and $\pi_1(\mathcal{D}(G))$ is trivial.  In particular, if $G$ is nonabelian and nilpotent, then $\pi_1(\mathcal{D}(G))$ is trivial.
 \end{lemma}
 \begin{proof}
 Lemma 3.2 of \cite{jensen2015character} concludes that under the hypotheses of this lemma, $\mathcal{G}(G)$ is connected and $\pi_1(\mathcal{G}(G))$ is trivial.  As $\mathcal{G}(G)$ is connected if and only if $\mathcal{D}(G)$ is connected by Theorem \ref{connectedness}, Theorem \ref{bijection} applies, yielding the result.  
 \end{proof}
 
 In order to restate the analogous theorem to the main result of \cite{jensen2015character}, we will introduce notation defined in \cite{benjamin1997coprimeness}.  In this work, the author states that a group $G$ satisfies property $P_k$ if every set of $k$ distinct elements of $\cd(G)$ is (setwise) relatively prime.  
 
 \begin{theorem}
 Suppose $G$ is a finite solvable group with $\mathcal{D}(G)$ connected.  Suppose that $k$ is the smallest positive integer such that $G$ satisfies property $P_k$, and suppose that $k \geq 3$.  Then the rank of $\pi_1(\mathcal{D}(G))$ is at most $k^2 - 3k + 1$.
 \end{theorem}
 \begin{proof}
 Theorem 3.7 of \cite{jensen2015character} states that if $G$ is a finite solvable group with $\mathcal{G}(G)$ connected of dimension $n$, then $\pi_1(\mathcal{G}(G))$ has rank at most $n^2 + n -1$.  If the dimension of $\mathcal{G}(G)$ is $n$, we conclude both that $\cd(G)$ has a subset $X$ satisfying $|X| = n+1$ and $\gcd(X) > 1$ and that no subset of $\cd(G)$ of size $n+2$ has a nontrivial common divisor.  We conclude that $k = n+2$, and plugging in $k-2$ for $n$ gives us that $\pi_1(\mathcal{G}(G))$ has rank at most $(k-2)^2 + (k-2) - 1 = k^2 - 3k + 1$.  By Theorem \ref{bijection}, the groups	 $\pi_1(\mathcal{G}(G))$ and $\pi_1(\mathcal{D}(G))$ are isomorphic, and therefore the bound $k^2 - 3k + 1$ applies to the rank of $\pi_1(\mathcal{D}(G))$ as well. 
\end{proof}
  
\section{Directions for Future Research}\label{research}

We know that the common and the prime divisor simplicial complex have isomorphic fundamental groups, but can this be extended? Could this be proven using only relationships between numbers, like the relationship between common and prime divisors, not simplicial complexes at all? If so, can we find other relationships between numbers that exhibit the same isomorphic fundamental group property? Another possible avenue for further research is extending this theory to higher homotopy and homology groups.

\section{Example of a Simplicial Complex}\label{example}

Given the set $X=\{22,33,65,91,210\}$, we form the common divisor simplicial complex $\mathcal{G}(X)$ as shown in Figure \ref{gx}. The set of prime divisors of $X$ is $P=\{2,3,5,7,11,13\}$, and from this set we form the prime divisor simplicial complex $\mathcal{D}(X)$ as shown in Figure \ref{dx}.

\begin{figure}[!ht]
\begin{minipage}[.5\textheight]{.5\textwidth}
\centering
\begin{tikzpicture}
\draw (0,0)--(-1,0.6)--(-1,-0.6)--(0,0)--(1,0.6)--(1,-0.6)--(0,0);
\filldraw[color = black, fill = black] (0,0) circle[radius = .03];
\filldraw[black] (0,0) node [label=$210$]{}
                 (-1,0.6) node [label=above:$22$]{}
				 (-1,-0.6) node [label=below:$33$]{}
                 (1,0.6) node [label=above:$65$]{}
				 (1,-0.6) node [label=below:$91$]{};
\end{tikzpicture}
\caption{$\mathcal{G}(X)$}\label{gx}
\end{minipage}%
\begin{minipage}[.5\textheight]{.5\textwidth}
\centering
\begin{tikzpicture}
\filldraw[fill=lightgray] (0,0)--(0,1)--(1,1)--(1,0)--(0,0);
\draw (0,1)--(-0.5,0.5)--(0,0)--(1,1)--(1.5,0.5)--(1,0);
\draw[dashed] (1,0)--(0,1);
\filldraw[black] (0,0) node [label=below:$3$]{}
                 (0,1) node [label=above:$2$]{}
                 (1,1) node [label=above:$5$]{}
                 (1,0) node [label=below:$7$]{}
                 (-0.5,0.5) node [label=left:$11$]{}
				 (1.5,0.5) node [label=right:$13$]{};
\end{tikzpicture}
\caption{$\mathcal{D}(X)$}\label{dx}
\end{minipage}
\end{figure}

\bibliographystyle{amsplain}
\bibliography{main}

\end{document}